\documentclass[12pt]{amsart}
\usepackage{amssymb,amsmath,amsthm,amsfonts} 
\usepackage{enumerate}
\usepackage {latexsym}
\usepackage{graphicx}
\usepackage{epsfig}
\usepackage{bbm}
\include{srctex}
\usepackage{marginnote}
\usepackage[letterpaper, top=1.2in, bottom=1.2in, outer=1.1in, inner=1.1in, heightrounded, marginparwidth=2.5cm, marginparsep=2mm]{geometry}
  
\allowdisplaybreaks
\newtheorem{theorem}{Theorem}[section]
\newtheorem{lemma}[theorem]{Lemma}
\newtheorem{prop}[theorem]{Proposition}
\newtheorem{cor}[theorem]{Corollary}
\theoremstyle{remark}

\def\R{{\mathbb R}}

\def\O{{\mathcal O}}
\def\norm[#1][#2]{\| #1\|_{#2}}
\def\bignorm[#1][#2]{\bigg\| #1 \bigg\|_{#2}}
\def\ptl[#1][#2]{\frac{\partial #1}{\partial #2}}
\def\japanese[#1]{\langle #1 \rangle}

\def\eps{\varepsilon}

\def\la{\langle}
\def\ra{\rangle}
\def\les{\lesssim}
\def\ges{\gtrsim}

\def\PV{{\rm P.V.}}

\def\f{\frac}
\def\mR{\mathcal R}
\renewcommand{\tilde}{\widetilde}
\def\be{\begin{equation}}
\def\ee{\end{equation}}
\numberwithin{equation}{section}

\begin{document}

\title[Strichartz estimates for Dirac operators]
{Limiting absorption principle and Strichartz estimates for Dirac operators in two and higher  dimensions  }
\author{M. Burak Erdo\smash{\u{g}}an, Michael Goldberg, William R. Green}
\thanks{The first author was partially supported by NSF grant DMS-1501041.
The second author is supported by Simons Foundation Grant 281057. The third author is supported by Simons Foundation Grant 511825  and acknowledges
the support of a Rose-Hulman summer professional development grant}  
\address{Department of Mathematics \\
University of Illinois \\
Urbana, IL 61801, U.S.A.}
 \email{berdogan@math.uiuc.edu}

\address{Department of Mathematics\\
University of Cincinnati \\
Cincinnati, OH 45221 U.S.A.}
\email{goldbeml@ucmail.uc.edu}
\address{Department of Mathematics\\
Rose-Hulman Institute of Technology \\
Terre Haute, IN 47803, U.S.A.}
\email{green@rose-hulman.edu}


\begin{abstract}
In this paper we consider
Dirac operators 
in $\R^n$, $n\ge2$, with a potential $V$. Under mild decay and continuity assumptions on $V$ and some spectral assumptions on the operator,  we prove a limiting absorption principle for the resolvent, which implies a family of Strichartz estimates for the linear Dirac equation. For large potentials the dynamical estimates are not an immediate corollary of the free case since the resolvent of the free Dirac operator does not decay  in operator norm on weighted $L^2$ spaces as the frequency goes to infinity.   
\end{abstract}

\maketitle

\section{Introduction}

In this paper we obtain limiting absorption principle bounds and Strichartz estimates for the linear Dirac equation in  dimensions two and higher with potential:
\begin{align}\label{dirac}
i\partial_t \psi(x,t) = (D_m +V(x)) \psi(x,t), \,\,\,\, \psi(x,0)= \psi_0(x).
\end{align} 
Here $x\in \R^n$ and $\psi(x,t) \in \mathbb{C}^{2^{N}}$ where $N=\lfloor \frac{n+1}{2}\rfloor$. The $n$-dimensional free Dirac operator $D_m$ is
 defined by
\begin{align}\label{eq:Dm}
   D_m= -i \alpha \cdot \nabla+ m\beta = -i \sum_{k=1}^{n} \alpha_k \partial_k + m\beta,
\end{align}
where $m\geq 0$ is a constant, and the $2^{N}\times 2^{N}$ Hermitian matrices $\alpha_j$ satisfy the anti-commutation relationships
\begin{align} \label{matrixeq}
	\left\{\begin{array}{ll}
		\alpha_j \alpha_k+\alpha_k\alpha_j =2\delta_{jk}
		\mathbbm 1_{\mathbb C^{2^{N}}}
		& j,k \in\{1,2,\dots, n\}\\
		\alpha_j \beta+\beta \alpha_j= 
		\mathbb O_{\mathbb C^{2^{N}}}\\
		\beta^2 = \mathbbm 1_{\mathbb C^{2^{N}}}
		\end{array}
	\right.
\end{align}
Physically, $m$ represents the mass of the quantum particle.  If $m=0$ the particle is massless and if $m>0$ the particle is massive.  We note that dimensions $n=2,3$ are of particular physical interest.  Following standard conventions, we define the free Dirac operator in dimension two with the Pauli spin matrices
$$
	\alpha_1=\left[\begin{array}{cc} 0 & -i\\ i & 0
	\end{array}\right],\
	\alpha_2=\left[\begin{array}{cc} 0 & 1\\ 1 & 0
	\end{array}\right],\
	\beta=\left[\begin{array}{cc} 1 & 0\\ 0 & -1
	\end{array}\right].	
$$
In dimension three we use
\begin{align*}
	\beta=\left[\begin{array}{cc} I_{\mathbb{C}^2} & 0\\ 0 & -I_{\mathbb{C}^2}
	\end{array}\right], \ \alpha_i=\left[\begin{array}{cc} 0 & \sigma_i \\ \sigma_i & 0
	\end{array}\right],
	\end{align*}
	\begin{align*}
	\sigma_1=\left[\begin{array}{cc} 0 & -i\\ i & 0
	\end{array}\right],\
	\sigma_2=\left[\begin{array}{cc} 0 & 1\\ 1 & 0
	\end{array}\right],\
	 \sigma_3=\left[\begin{array}{cc} 1 & 0\\ 0 & -1
	\end{array}\right].
\end{align*}
In higher dimensions $n>3$, one can create a full set of anti-commuting matrices $\alpha_j$ iteratively, see \cite{KS} for an explicit construction.

The Dirac equation arose as an attempt to reconcile the theories of relativity and quantum mechanics and describe the behavior of subatomic particles at near luminal speeds.  The relativistic relationship between energy, momentum and mass, $E^2=c^2p^2+m^2c^4$ can be combined with the quantum-mechanical notions of energy $E=i\hbar \partial_t$ and momentum $p=-i\hbar \nabla$ to obtain a Klein-Gordon equation  
$$
	-\hbar^2 \partial_{t}^2 \psi(x,t)=-c^2\hbar^2 \Delta \psi(x,t)
		+m^2c^4 \psi(x,t).
$$
Here $\hbar$ is Planck's constant and $c$ is the speed of light.  
However, the Klein-Gordon does not preserve $L^2$ norm of the initial data and is incompatible with quantum mechanical interpretations of the wave function.
By considering $E$ directly, one arrives at the non-local equation
\be\label{eq:nonloc}
i\hbar \partial_t \psi(x,t)=\sqrt{-c^2\hbar^2 \Delta 
	+m^2c^4}\, \psi(x,t).
\ee
In our mathematical analysis, we rescale so that we may take the constants $\hbar$ and $c$ to be one.    Dirac's insight was to rewrite the right hand side in terms of the first order operator $D_m=-i\alpha \cdot \nabla+m\beta$. This leads to the free Dirac equation, \eqref{dirac} with $V=0$, a system of coupled hyperbolic equations with $\alpha, \beta$ required to be matrices.  Dirac's modification allows one to account for the spin of quantum particles, as well as providing a way to incorporate external electro-magnetic fields in a manner compatible with the relativistic theory where the Klein-Gordon and \eqref{eq:nonloc} cannot.  In addition, we note that \eqref{eq:nonloc} has infinite speed of propagation, which is in contrast with the causality principle in relativity.  In dimension $n=3$, the Dirac equation models the evolution of spin $1/2$ particles, while in dimension $n=2$ the massless Dirac equation is of considerable interest due to its connection to graphene, see, e.g., \cite{FW}.  

Formally, the Dirac equation is a square root of a system of Klein-Gordon or wave equations when $m>0$ and $m=0$ respectively.  One consequence is that the spectrum of the free Dirac operators is unbounded in both the positive and negative directions.  In particular, the continuous spectrum of $D_m$ is $(-\infty,-m]\cup[m,\infty)$.  By Weyl's criterion, the continuous spectrum of the perturbed Dirac operator is also $(-\infty,-m]\cup[m,\infty)$ for a large class of potentials.   The absence of embedded eigenvalues in the continuous spectrum in general dimensions was established in \cite{BC1} for the  class of potentials we are interested in by adapting the argument of \cite{BG1} for three dimensions.  This result was used to study linearizations about a solitary wave for a non-linear equation.  For other results in this direction for small dimensions and specific classes of potentials see \cite{Roze,BG1,Vog,GM}. 
Finally, there is no singular continuous spectrum, see \cite{GM}. For a further background on the Dirac equation see \cite{Thaller}.

We denote the perturbed Dirac operator by $H:=D_m+V$, then $e^{-itH}$ is formally the solution operator to \eqref{dirac}.  For the class of potentials considered in Theorem~\ref{thm:main}, we note that $H$ is self-adjoint by the Kato-Rellich theorem.  We  denote $a-:=a-\eps$ for a small, but fixed $\eps>0$.  Further, we write $A\les B$ to indicate there is a fixed absolute constant $C>0$ so that $A\leq C B$.
 
\begin{theorem}
\label{thm:main} Let  $V$ be a  $2^{N}\times 2^{N}$ real Hermitian matrix  for all
$x\in\R^n$, $n\ge2$, with continuous entries satisfying  
 $|V_{ij}(x)|  \les  \la x\ra^{-1- } $ when $m=0$, and $|V_{ij}(x)|  \les  \la x\ra^{-2- } $ when $m>0$.
Furthermore, assume that threshold energies are regular.  Then, with $P_c$ being the projection onto the continuous
spectrum,
\begin{equation}\label{eq:strichmassive} \|\la \nabla\ra^{-\theta} e^{-itH} P_c f\|_{L_t^p(L_x^q)}
\les \|f\|_{L^2(\R^n)} \end{equation} in the case $m > 0$, provided that 
\begin{equation*}
\theta \geq \frac12 + \frac1p - \frac1q \quad \text{and} \quad 
\frac{2}{p}+\frac{n}{q}=\frac{n}{2}, \,\,\,\,\, 2\le q<\frac{2n}{n-2}.
\end{equation*}
In the case $m=0$, the bound is
\begin{equation}\label{eq:strichmassless} \| |\nabla|^{-\theta} e^{-itH} P_c f\|_{L_t^p(L_x^q)}
\les \|f\|_{L^2(\R^n)} \end{equation} provided that 
\begin{equation*}
\theta = \frac{n}2-\frac1p-\frac{n}q \quad \text{and} \quad  
\frac{2}{p}+\frac{n-1}{q}\leq\frac{n-1}{2}, \,\,\,\,\, p > 2,\  2 \leq q < \infty.
\end{equation*}
\end{theorem}

The combinations of $(p,q , \theta)$ stated above are the same ones found in Strichartz estimates for the free massive ($m>0$) and massless ($m=0$) Dirac equation, respectively.  Note that the range of admissible Strichartz exponents $(p,q)$ match those for the Schr\"odinger equation in the massive case, and the derivative is not homogeneous.  This reflects the fact that the low energy behavior of the Dirac system is comparable to the Schr\"odinger equation, while the high energy behavior is closer to the wave equation (which requires differentiability of initial data).  See the Appendix of~\cite{DF07} for a derivation of Strichartz estimates for the free evolution $e^{-itD_m}$.  The free massless Dirac system has the same scaling properties and admissible combinations as the free wave equation, which are proved in~\cite{KT} for the wave equation.

These families of perturbed Strichartz estimates are a consequence of the uniform resolvent estimates that we prove.  Much of the paper is devoted to proving the following resolvent bounds, which hold for any subset of the continuous spectrum of $H$.

\begin{theorem} \label{thm:limap}  Let  $V$ be a  $2^{N}\times 2^{N}$ real Hermitian matrix  for all
$x\in\R^n$, $n\ge2$, with continuous entries satisfying  
 $|V_{ij}(x)|  \les  \la x\ra^{-1- }$.
 Then for  $\sigma > \frac12$ there exists $\lambda_1<\infty$ so that 
\begin{equation} \label{eq:limap1}
\sup_{|\lambda| > \lambda_1 } \ 
\| \la x\ra^{-\sigma}   (H - (\lambda + i0))^{-1}
     \la x\ra^{-\sigma}\|_{2\to2} \les 1.
\end{equation} 
Under the assumption that the threshold energies are regular, and if $m>0$ the stronger decay condition $|V_{ij}(x)| \les \la x\ra^{-2-}$, this bound can be extended as follows 
$$
\sup_{|\lambda| >m} \ 
\| \la x\ra^{-\sigma}   (H - (\lambda + i0))^{-1}
     \la x\ra^{-\sigma}\|_{2\to2} \les 1, 
$$
provided that $\sigma>\frac12$ when  $m=0$, and   $\sigma>1$ when $m>0$.
\end{theorem}
We note that the proof of the high energy limiting absorption principle \eqref{eq:limap1} does not require $V$ to be real or Hermitian.     Since $V$ is assumed to be bounded and the free Dirac operator $D_m$ is self-adjoint, $H$ has the same domain as $D_m$ and for unit functions $\eta$ in the domain the quadratic form $\la H\eta, \eta\ra$ is confined to a strip of finite width around the real axis.  

The immediate consequence of \eqref{eq:limap1} is that there cannot be any embedded eigenvalues or resonances on $(-\infty,-\lambda_1)\cup(\lambda_1,\infty)$ for $\lambda_1=\lambda_1(V) $ sufficiently large.  A perturbation argument shows that the eigenvalue-free zone
extends to a sector of the complex plane.
\begin{cor} \label{cor:e-free}
Under the hypotheses of Theorem~\ref{thm:limap}, there exist $\lambda_1 < \infty$ and $\delta > 0$ depending on 
$V$, $m$, and $\sigma > \frac12$ so that
\be \label{eqn:limap2}
\sup_{\substack{|\lambda| > \lambda_1\\ 0 < |\gamma| < \delta |\lambda|}} 
\| \la x\ra^{-\sigma} (H - (\lambda + i\gamma))^{-1} \la x\ra^{-\sigma}\|_{2\to 2} \les 1.
\ee 
As a result, there is a compact subset of the complex plane outside of which the spectrum of $H$ is confined to the real axis.
\end{cor}

Our results apply to a broad class of electric potentials $V(x)$ and require no implicit smallness condition, only that $V$ is bounded, continuous and satisfies a mild polynomial decay at infinity.  The potentials need not be small, radial, or smooth.  Our results apply for the  potentials   that naturally arise when linearizing about soliton solutions for the non-linear Dirac equation.

There is a rich history of results on limiting absorption principles and mapping estimates of dispersive equations.  Much of this history is focused on the analysis of the Schr\"odinger, wave or Klein-Gordon equation.  We refer the reader to \cite{GV,RS,GS,GST,DF07,Stef,EGS,BT,MMT,EGS2,FV,DFVV,RT2}, for example.  There are far fewer results in the case of the Dirac system, due to its more complicated mathematical structure.  

It is known that the Dirac resolvent does not decay in the spectral parameter, \cite{Yam}.  That is, the bound \eqref{eq:limap1} does not decay as $\lambda\to \infty$.  This is a stark contrast to the Schr\"odinger resolvent in which one obtains a decay in the spectral parameter $\lambda$ as $\lambda \to \infty$. 
The bootstrapping argument of Agmon, \cite{Agmon}, produces uniform bounds on the resolvent operators only on compact subsets of the purely absolutely continuous spectrum.   Limiting absorption principles have been studied to establish the limiting behavior of resolvents as one approaches the spectrum, see \cite{Vog,BaHe,BG1}.  The work of Georgescu and Mantoiu provides resolvent bounds on compact subsets of the spectrum, \cite{GM}.  Other limiting absorption principles have been established, often in service of providing dispersive, smoothing or Strichartz estimates, \cite{Bouss1,DF,BG}.  Very recently, \cite{CG}, established a limiting absorption principle for the free massless Dirac operator in dimensions $n\geq 2$.

One consequence of the resolvent bounds in Theorem~\ref{thm:limap} is the family of Strichartz estimates given in Theorem~\ref{thm:main}.   Strichartz estimates have been used to study non-linear Dirac equations, \cite{MMNO,CTS,BH3,BH,BC1,BC2}.  These are often adapted to the problem by localizing in frequency or considering specialized potentials.  Strichartz estimates may be obtained by establishing a virial identity see, for example, \cite{BDF,C}, which consider magnetic potentials with a certain smallness condition.   The first and third author proved a class of Strichartz estimates for the two-dimensional Dirac equation, \cite{EGDirac2d}, by first establishing dispersive estimates of the two-dimensional Dirac propagator.

The paper is organized as follows:  In Section~\ref{sec:setup} we show how the Strichartz estimates in Theorem~\ref{thm:main} follows from the resolvent bounds in Theorem~\ref{thm:limap}.  The bulk of the paper is then devoted to proving Theorem~\ref{thm:limap}.

In Section~\ref{sec:free} we present the basic properties of the free resolvents of Dirac and Schr\"odinger operators. 
The small energy case of Theorem~\ref{thm:limap} is then treated
in Section~\ref{sec:small}. 
In Section~\ref{sec:high_limap}, we treat the case of large energies
by adapting an intricate argument originally devised in~\cite{EGS,EGS2}
for Schr\"odinger operators in dimensions $n \geq 3$
with a non-smooth magnetic potential.
A brief argument in Section~\ref{sec:complex} derives Corollary~\ref{cor:e-free} from
the main high-energy bounds.

\section{The basic setup}
\label{sec:setup}

The Strichartz estimates stated in Theorem~\ref{thm:main} will be proved using
Proposition~\ref{prop:Katoth}
below, which is essentially Theorem~4.1 in \cite{RS}. It is based on
Kato's notion of smoothing
operators, see~\cite{Kato}. We recall that for a self-adjoint operator $H$,
an operator $\Gamma$ is called
$H$-smooth in the sense of Kato if for any $f\in {\mathcal D}(H)$
\begin{equation}
\label{eq:Csmooth} \|\Gamma e^{-it H} f\|_{L^2_t L^2_x}\le
C_{\Gamma}(H) \|f\|_{L^2_x}.
\end{equation}
  Let $\Omega\subset \R$ and let $P_\Omega$ be a
spectral projection of $H$ associated with a set $\Omega$. We say
that  $\Gamma$ is $H$-smooth on $\Omega$ if $\Gamma P_{\Omega}$ is
$H$-smooth.  It is not difficult to show (see
e.g.~\cite[Theorems XIII.25 and XIII.30]{RS4}) that,    $\Gamma$ is $H$-smooth on
$\Omega$ if
\begin{equation}
\label{eq:smon} \sup_{\lambda \in  \Omega} \| \Gamma
[R^+_{H}(\lambda ) -R_H^-(\lambda) ] \Gamma^* \|_{L^2 \to L^2 } \le C_{\Gamma}(H,
\Omega).
\end{equation}

Given the known Strichartz bounds for the free Dirac equation, the following proposition and Theorem~\ref{thm:limap} imply  Theorem~\ref{thm:main}.  For brevity we state only the $m>0$ case. 
\begin{prop}\label{prop:Katoth}
Let $H_0=D_m$, $m>0$, and $H= H_0 + V$, where $|V(x)|\lesssim \la x\ra^{-2\sigma}$. Assume that $w(x):=\la x \ra^{-\sigma}$
 is $H_0$-smooth  and $H$-smooth on $\Omega$  for some $\Omega\subset \R$.
Assume also that the unitary semigroup $e^{-it H_0}$ satisfies the
estimate
\begin{equation}
\label{eq:strH0} \big\| \la\nabla\ra^{-\theta} e^{-it H_0}   \big\|_{L^2 \to L^q_t L^r_x} < \infty 
\end{equation}
for some $q\in (2,\infty]$, $r\in [1,\infty]$, and $\theta \in \R$. Then the semigroup
$e^{-itH}$ associated with $H = H_0 + V$, restricted to the spectral
set $\Omega$, also verifies the estimate \eqref{eq:strH0}, i.e.,
\begin{equation}
\label{eq:strHk}  \big\|\la\nabla\ra^{-\theta} e^{-it H}P_\Omega  \big\|_{L^2 \to L^q_t L^r_x}  < \infty.  
\end{equation}
\end{prop}
\begin{proof} For completeness we supply the proof following \cite{RS}.  
We have 
$$
e^{-it H}P_\Omega f = e^{-it H_0}P_\Omega f - i \int_0^t e^{-i (t-s) H_0} V   e^{-is H }P_\Omega f ds.
$$ 
By Christ-Kiselev Lemma \cite{CK}, it suffices to prove that 
 $$
 \Big\|\la\nabla\ra^{-\theta} \int_0^\infty e^{-i (t-s) H_0} V   e^{-is H }P_\Omega  f ds \Big\|_{  L^q_t L^r_x}\lesssim \|f\|_{L^2}. 
 $$
 Using \eqref{eq:strH0}, we bound the left hand side by
\be\label{eq:temp11}
 \Big\| \int_0^\infty e^{i s H_0} V w^{-1}  w    e^{-is H }P_\Omega f ds \Big\|_{L^2 }.  
 \ee
 Since $w$ is $H_0$ smooth and $H$-smooth on $\Omega$, and $|Vw^{-1}|\les w$, 
 we have 
 $$
  \|we^{-itH}P_\Omega f \|_{L^2_t L^2_x}\lesssim \|f\|_{L^2_x}, \,\,\,\,\, \|Vw^{-1} e^{-it H_0} f\|_{L^2_t L^2_x}\lesssim \|f\|_{L^2_x},
 $$
 and its dual 
 $$
 \Big\| \int_0^\infty e^{i s H_0} V w^{-1} g(s,x) ds \Big\|_{L^2} \les \|g\|_{L^2_sL^2_x}.
 $$
Composing these two inequalities suffices to bound \eqref{eq:temp11}  by $\|f\|_{L^2}$.
 \end{proof}

\section{Properties of the Free Resolvent} \label{sec:free}

 The following identity,\footnote{Here and throughout the paper, scalar operators such as  $-\Delta+m^2-\lambda^2$ are understood as $(-\Delta+m^2-\lambda^2)\mathbbm 1_{\mathbb C^{2^{N}}}$.}  which follows from   \eqref{matrixeq},
\be  \label{dirac_schro_free}
	(D_m-\lambda \mathbbm 1)(D_m+\lambda \mathbbm 1) =(-i\alpha\cdot \nabla +m\beta -\lambda \mathbbm 1)
	(-i\alpha\cdot \nabla+m\beta+\lambda \mathbbm 1)   =(-\Delta+m^2-\lambda^2) 
\ee
allows us to formally define the free Dirac resolvent
operator $\mathcal R_0(\lambda)=(D_m-\lambda)^{-1}$ in terms of the
free resolvent $R_0(\lambda)=(-\Delta-\lambda)^{-1}$ of  the Schr\"odinger operator for $\lambda$ in the resolvent set:
\begin{align}\label{eqn:resolvdef}
	\mathcal R_0(\lambda)=(D_m+\lambda) R_0(\lambda^2-m^2).
\end{align}
We first discuss the properties of Schr\"odinger resolvent $R_0$. There are two possible continuations to the positive halfline, namely
$$
R_0^\pm(\lambda^2)=\lim_{\eps\to 0^+} R_0(\lambda^2\pm i\eps),\,\,\,\,\,\,\lambda>0,
$$
where the limit is in the operator norm from $L^2_\sigma $ to $L^2_{-\sigma}$, $\sigma>\frac12$. 
Here $L^2_\sigma$ denotes the weighted $L^2$ space with norm
$$
\|f\|_{L^2_\sigma} := \| \la \cdot\ra^{\sigma} f  \|_{L^2}.
$$  
Existence of the limits $R_0^\pm(\lambda^2)$ is known as the limiting absorption principle.
In fact $R_0^\pm(\lambda^2)$ varies continuously in $\lambda$ over the  interval  $(0,\infty)$.
In dimensions $n\geq 3$ the continuity extends to $\lambda \in [0,\infty)$ with a uniform bound
\be\label{eq:lapScunif}
 \|R_0^\pm(\lambda^2)\|_{L^2_\sigma\to L^2_{-\sigma}}\leq C_{\sigma,n}, \,\,\,\,\,\lambda \geq 0
\ee 
provided $\sigma > 1$.
In two dimensions the free Schr\"odinger operator has a threshold resonance and consequently $R_0^\pm(\lambda^2)$ is unbounded as $\lambda$ approaches zero.
However there is still a useful uniform estimate,
\be\label{eq:lapSc}
\|\nabla R_0^\pm(\lambda^2)\|_{L^2_\sigma\to L^2_{-\sigma}}+ \lambda\|R_0^\pm(\lambda^2)\|_{L^2_\sigma\to L^2_{-\sigma}}\leq C_{\sigma,n}, \,\,\,\,\,\lambda >0,\,\,\,\sigma>\frac12, 
\ee
which is true in all dimensions $n\geq 2$.  This bound for large $\lambda$ is largely due to scaling
considerations.  The bound for small $\lambda$ will be proved in the next section.

Using the limiting absorption bounds \eqref{eq:lapSc} for Schr\"odinger and \eqref{eqn:resolvdef}, we obtain for $n\geq 2$
\be\label{eq:freeDiracLAP}
\|\mathcal R_0(\lambda)\|_{L^2_\sigma\to L^2_{-\sigma}} \leq C_{\sigma,\lambda_0,n}, \,\,\,\,\,|\lambda| >\lambda_0>m,\,\,\,\sigma>\frac12. 
\ee
An analogous uniform bound holds on the entire interval $|\lambda|>m$ if $n\geq 3$ and $\sigma > 1$. 
In the case $m=0$ we have the following stronger uniform bound for $n\geq 2$
\be\label{eq:masslessLAP}
\|\mathcal R_0(\lambda)\|_{L^2_\sigma\to L^2_{-\sigma}} \leq C_{\sigma, n}, \,\,\,\,\,|\lambda|  >0,\,\,\,\sigma>\frac12. 
\ee
 In particular, two dimensional massless free Dirac operator does not have a threshold resonance.
 
The kernel of the free resolvent $R_0^+(\lambda^2)$ in $\R^n$ is
given by\footnote{Constants $C_n$ are allowed to change from line to
line.}
\[
R_0^+(\lambda^2)(x,y) = C_n\,
\frac{\lambda^\frac{n-2}{2}}{|x-y|^{\frac{n-2}{2}}}
H_{\frac{n-2}{2}}^+(\lambda|x-y|)
\]
where $H_\nu^+$ is a Hankel function. There is the scaling relation
\begin{equation}   \label{eq:Rscale}
R_0^+(\lambda^2)(x,y)= \lambda^{n-2} R_0^+(1)(\lambda x,\lambda y)
\quad \forall\;\lambda>0
\end{equation}
and the representation, see the asymptotics of $H_\nu^+$
in~\cite{AbSteg},
\begin{equation}
  \label{eq:Hsplit}
  R_0^+(1) (x,y) = \frac{e^{i|x-y|}}{|x-y|^{\frac{n-1}{2}}}
   a(|x-y|) + \frac{b(|x-y|)}{|x-y|^{n-2}}
\end{equation}
provided $n\ge2$.  Here
\begin{equation}\label{eq:a}
 |a^{(k)}(r)| \les r^{-k} \quad\forall\;k\ge0, \quad a(r)=0 \quad\forall\; 0< r<\f12
 \end{equation}
and $b(r)=0$ for all $r>\f34$, with
\begin{align}
&  |b^{(k)}(r)| \les 1 \quad\forall\;k\ge0, \quad n\text{\ \
odd} \label{eq:bodd}\\
&\left. \begin{array}{lll} |b^{(k)}(r)| &\les 1 &
\forall\;0\le k< n-2 \\
|b^{(k)}(r)| &\les r^{n-k-2}|\log r| &\forall\;k\ge n-2
\end{array}\right\}
 \; n\ge 2\text{\ \ even} \label{eq:beven}
\end{align}
for all $r>0$. 
In dimension $n=2$ we will need the following more detailed expansion of $b$
\be \label{eq:2db}
b(r) = \bigg(-\frac{1}{2\pi}\big(\log(r/2) + \gamma\big) + \frac{i}{4}\bigg)  + \mathcal E(r),
\ee
where  
$$
	\big|  \mathcal E^{(k)}(r) \big| \les \bigg| \frac{d^k}{dr^k} (r^2\log r) \bigg| \qquad k=0,1,2.
$$

In order to gain sharp control over the scaling behavior as $\lambda \to \infty$
 we discuss  the $\sigma=\frac12$ endpoint of the limiting absorption principle.   As in Chapter~XIV of~\cite{Hor} define
\[
\|f\|_{B}:=\sum_{j=0}^\infty 2^{\frac{j}{2}} \|f\|_{L^2(D_j)}, \quad
\|f\|_{B^*} := \sup_{j\ge0} 2^{-\frac{j}{2}} \|f\|_{L^2(D_j)},
\]
where   $D_j=\{ x\::\: |x|\sim 2^j\}$ for $j\ge1$ and $D_0=\{|x|\le 1\}$. 
For each $\sigma > \frac12$, there are containment relations $L^2_\sigma \subset B$
and $B^* \subset L^2_{-\sigma}$.
It is known that $\|R_0(1)\|_{B\to B^*} <\infty$. 

Note that
\be\label{eq:BstarB}
\|Vf\|_B\lesssim \|f\|_{B^*} \sum_{j=0}^\infty 2^j \|V\|_{L^\infty(D_j)} \les \|\la x\ra^{1+} V\|_{L^\infty} \|f\|_{B^*}.
\ee
Also recall that, by Lemma~3.1 in \cite{EGS2}, we have the following scaling relations for any $\lambda\ge 1$ 
  \[
\|f( \lambda^{-1}\cdot) \|_B \les \lambda^{\frac{n+1}{2}}
\|f\|_B,\quad \|g(\lambda \cdot) \|_{B^*} \les
\lambda^{-\frac{n-1}{2}} \|g\|_{B^*},
  \]
provided the right-hand sides are finite.  This   and  \eqref{eq:Rscale} immediately
imply the following statement. In what follows, $R_0$ stands for
either of $R_0^{\pm}$.

\begin{prop}
  \label{prop:R0scale}    For all $\lambda\ge1$, we have 
  \[
\|R_0(\lambda^2)\|_{B\to B^*} \les \lambda^{-1} \|R_0(1)\|_{B\to
B^*}.
  \] 
\end{prop}
\begin{proof}
  First, from \eqref{eq:Rscale}
  \[
 (R_0^+(\lambda^2) f)(x) = \lambda^{-2} [R_0^+(1)
 f(\cdot\lambda^{-1})] (\lambda x)
  \]
  Hence, by the previous lemma,
  \[
\|R_0^+(\lambda^2) f\|_{B^*} \les \lambda^{-2}
\lambda^{-\frac{n-1}{2}} \| R_0^+(1)
 f(\cdot\lambda^{-1}) \|_{B^*}
 \les \lambda^{-1} \|R_0^+(1)f\|_{B^*}
  \]
  as claimed.
\end{proof}

\section{Energies close to $\{-m,m\}$} \label{sec:small}

In this section, assuming the regularity of  threshold energies, we prove 
Theorem~\ref{thm:limap} when the spectral parameter $\lambda$ is sufficiently close to the threshold energy $\lambda=m$, respectively $\lambda=-m$.  We consider the positive portion of the spectrum $[m,\infty)$, the negative part can be controlled similarly.
That is,  for sufficiently small $\lambda_0$ 
\begin{align}\label{eq:sm1}
&\sup_{m<\lambda<\lambda_0} \|w (\mR^+_V(\lambda)-\mR^-_V(\lambda)) w \|_{2\to2} < \infty,
\end{align}
where $w=\la x\ra^{-\sigma} $ for   $\sigma >1$. In fact, we prove that
\begin{align}\label{eq:sm2}
&\sup_{m<\lambda<\lambda_0} \|w  \mR^\pm_V(\lambda)  w \|_{2\to2} < \infty,
\end{align}
provided that $\sigma > 1$ when $m > 0$, and provided $\sigma > \frac12$ when $m=0$.
A similar statement holds for negative energies. 

We refer to reader to \cite{EGDirac2d} for the case  $n=2$ and $m>0$, as this argument is substantially different from the other cases.  In all the remaining cases we have
\begin{multline*}
  \|w \mR^\pm_V(\lambda) w \|_{2\to2}  =  \|w (1+\mR_0^\pm(\lambda )V)^{-1}w^{-1} w \mR^\pm_0(\lambda )
w
 \|_{2\to2} \\
 \le  \|w (1+\mR_0^\pm(\lambda )V)^{-1}w^{-1} \|_{2\to2}  \| w \mR^\pm_0(\lambda )
w
 \|_{2\to2},
\end{multline*}
so it suffices to show that
\begin{align}\label{eq:Z0small}
&\sup_{m< \lambda<\lambda_0}  \|w (1+\mR_0^\pm(\lambda )V)^{-1}w^{-1} \|_{2\to2} <\infty, \\
&\sup_{m< \lambda<\lambda_0} \| w \mR^\pm_0(\lambda )w\|_{2\to2} < \infty. \label{eq:R0small}
\end{align}
In dimensions $n\geq 3$, \eqref{eq:R0small}  is an immediate consequence of the fact that the free Dirac operator is regular at the threshold, provided $\sigma > 1$.  We will show below that \eqref{eq:R0small} is also true when $m=0$ and $n\geq 2$ with
$\sigma > \frac12$.  The $n=2$ case is somewhat surprising because the threshold is not regular for the free Schr\"odinger operator, nor for the free Dirac operator with $m > 0$.

In dimensions $n\geq 3$, let $G=\mR_0^\pm(m)= (D_m+m)R_0(0)$.  In the case $n=2$, $m=0$, define $G= D_0 G_0$,
where 
\begin{equation}
	G_0f(x)=-\frac{1}{2\pi}\int_{\R^2} \log|x-y|f(y)\,dy \label{G0 def}. 
\end{equation}
Let $B_\lambda^\pm=\mR_0^\pm(\lambda)-G$. We assume that the threshold $m$ is a regular point of the spectrum, namely the boundedness of the operators
\be
\label{eq:z0GL}  w(I + GV)^{-1}w^{-1} = (I+wGVw^{-1})^{-1}: L^2\to L^2.
\ee
By a standard Fredholm alternative argument, \eqref{eq:z0GL} is equivalent to the absence of resonances and eigenfunctions at $m$.  
We now prove that  
under suitable conditions
\be 
\label{eq:z0BL}  \|wB^\pm_\lambda Vw^{-1}\|_{2\to2} \les \|wB^\pm_\lambda  w \|_{2\to2}  \rightarrow
0\;\;\;\;\;\text{ as } \lambda \to m^+.
\ee
This and \eqref{eq:z0GL} imply \eqref{eq:Z0small} by summing the Neumann series directly,
and it implies \eqref{eq:R0small} since $w G w$ is $L^2$-bounded for $\sigma > 1$ if $m > 0$
and $\sigma > \frac12$ if $m=0$.

To prove  the   bound  \eqref{eq:z0BL} recall the  properties of the kernel of $B^\pm_\lambda$:
with $\lambda =\sqrt{m^2+z^2}$, we have 
\begin{multline}\label{eq:Bexp}
B_\lambda^\pm=\mR_0^\pm(\lambda)-G=\big(D_m+\sqrt{m^2+z^2}\big)R_0(z^2) - (D_m+m )R_0(0) \\
=\big(\sqrt{m^2+z^2} -m\big)R_0(z^2) + m (\beta+I) \big[R_0(z^2)- R_0(0)\big] -i\alpha\cdot\nabla \big[R_0(z^2)- R_0(0)\big]
\end{multline}
in dimensions $n \geq 3$.  When $m=0$ we have
\be \label{eq:Bexpm0}
B_\lambda^{\pm} = zR_0(z^2) - i\alpha \cdot \nabla\big[R_0(z^2) - R_0(0)\big] 
\ee
and this holds for $n=2$ by replacing $R_0(0)$ with $G_0$.

By the limiting absorption principle for the free Schr\"odinger operator, the second summand of \eqref{eq:Bexp} goes  to zero as $z\to 0$ as  operators from   $L^2_\sigma $ to $L^2_{-\sigma}$,   provided that $\sigma>1$, see \eqref{eq:lapScunif}. This can also be proved using the limiting absorption bound \eqref{eq:lapSc} at frequency 1 and scaling, similar to the remaining cases that we discuss below.  The remaining terms are identical to those in \eqref{eq:Bexpm0} or better, since $0\leq \sqrt{m^2 + z^2} - m \leq z$.  We will prove that both terms of \eqref{eq:Bexpm0} go to zero for $\sigma>\frac12$ for dimensions $n\geq 2$.

For the first term, using the scaling relation \eqref{eq:Rscale} and the representation \eqref{eq:Hsplit} we have 
$$
z R_0(z^2) 
=    zR_0(z^2)(x,y)\widetilde\chi(z|x-y|) + z\frac{b(z|x-y|) }{|x-y|^{n-2}},
$$
where $\widetilde\chi$ is a smooth cutoff for the complement of the unit ball.
Using \eqref{eq:bodd} and \eqref{eq:beven}, the low energy term  can be bounded as follows
$$\Big|   z\frac{b(z|x-y|) }{|x-y|^{n-2}}   \Big| \les z\frac{(z|x-y|)^{0-}}{|x-y|^{n-2}} \chi(z|x-y|)\les z\frac{(z|x-y|)^{-1+}}{|x-y|^{n-2 }}= \frac{z^{0+}}{|x-y|^{n-1- }}. 
$$
By the weighted version of the Schur's test, this operator is $O(z^{0+}) $ as $z\to 0$ as an operator from   $L^2_\sigma $ to $L^2_{-\sigma}$,   provided that $\sigma>1/2$. 
We can rewrite the high energy term using the scaling relation \eqref{eq:Rscale}:
$$
 zR_0(z^2)(x,y)\widetilde\chi(z|x-y|) = z^{n-1}   \big[R_0(1)\widetilde \chi\big](zx,zy).
$$
Therefore, with $\chi$ be a smooth cutoff for $1/10$ neighborhood of the origin,
\begin{align*}
\big\|\la x\ra^{-\sigma}  z\Big[R_0(z^2)&(x,y)\widetilde\chi(z|x-y|) \Big] \la y\ra^{-\sigma}\big\|_{L^2\to L^2} \\ 
&\les  z^{-1}\big\| \la x/z\ra^{-\sigma}   \big[R_0(1)\widetilde \chi\big](x,y) \la y/z\ra^{-\sigma}\big\|_{L^2\to L^2}\\ 
&\les z^{2\sigma-1} \big\| (z+|x|)^{-\sigma}   \big[R_0(1)\widetilde \chi\big](x,y) (|y|+z)^{-\sigma}\big\|_{L^2\to L^2} \\
&\les  z^{2\sigma-1} \big\| \la x\ra^{-\sigma}   \big[R_0(1)\widetilde \chi\big](x,y) \la y\ra^{-\sigma}\big\|_{L^2\to L^2} \\
& \qquad + z^{2\sigma-1} \big\| (z+|x|)^{-\sigma} \chi(|x|)   \big[R_0(1)\widetilde \chi\big](x,y) \la y\ra^{-\sigma}\big\|_{L^2\to L^2}\\
& \qquad \qquad + z^{2\sigma-1} \big\| \la x\ra^{-\sigma}    \big[R_0(1)\widetilde \chi\big](x,y) (z+|y|)^{-\sigma} \chi(|y|) \big\|_{L^2\to L^2}.
\end{align*}
The first summand converges to zero provided that $\sigma>\frac12$ by the limiting absorption bound for the free Schr\"odinger operator for $\lambda=1$. 
Using the representation \eqref{eq:Hsplit} and the bound \eqref{eq:a}, and considering the Hilbert Schmidt norms, the second and third summands can be bounded by the square root of 
$$
z^{4\sigma-2}\int \la x\ra^{-2\sigma}\frac{1}{\la x-y\ra^{ n-1} } (z+|y|)^{-2\sigma} \chi(|y|) dx dy \les z^{0+} \to 0,
$$
provided that $\sigma>\frac12$.

 We now consider  the second summand in \eqref{eq:Bexpm0}.  In dimensions $n \geq 3$ we may use the scaling relation \eqref{eq:Rscale} and the representation \eqref{eq:Hsplit} to write 
$$
 \nabla \big[R_0(z^2)- R_0(0)\big] 
= \nabla \Big[R_0(z^2)(x,y)\widetilde\chi(z|x-y|) + \frac{b(z|x-y|)-b(0)}{|x-y|^{n-2}}  \Big],
$$
where $\widetilde\chi$ is a smooth cutoff for the complement of the unit ball.
Using \eqref{eq:bodd} and \eqref{eq:beven}, the low energy term  can be bounded as follows
$$\Big| \nabla  \frac{b(z|x-y|)-b(0)}{|x-y|^{n-2}}   \Big| \les \frac{z^{0+}}{|x-y|^{n-1-}}, 
$$ 
which goes to zero as $z\to 0$ as an operator from   $L^2_\sigma $ to $L^2_{-\sigma}$,   provided that $\sigma>1/2$.  If $n = 2$ we use \eqref{eq:2db}  
to claim an analogous bound
\begin{equation*}
\big| \nabla\big[ b(z|x-y|) - G_0\big]  \big| \les \frac{z^{0+}}{|x-y|^{1-}}.
\end{equation*}
Turning our attention to the high energy term, we use the scaling relation \eqref{eq:Rscale} to write 
$$
\nabla \Big[R_0(z^2)(x,y)\widetilde\chi(z|x-y|) \Big]= z^{n-1} \nabla\big[R_0(1)\widetilde \chi\big](zx,zy).
$$
Therefore
\begin{align*}
\big\|\la x\ra^{-\sigma}\nabla \Big[R_0(z^2)&(x,y)\widetilde\chi(z|x-y|) \Big] \la y\ra^{-\sigma}\big\|_{L^2\to L^2} \\ 
&\les  z^{-1}\big\| \la x/z\ra^{-\sigma}  \nabla\big[R_0(1)\widetilde \chi\big](x,y) \la y/z\ra^{-\sigma}\big\|_{L^2\to L^2}\\ 
&\les z^{2\sigma-1} \big\| (z+|x|)^{-\sigma}  \nabla\big[R_0(1)\widetilde \chi\big](x,y) (|y|+z)^{-\sigma}\big\|_{L^2\to L^2} \\
&\les  z^{2\sigma-1} \big\| \la x\ra^{-\sigma}  \nabla\big[R_0(1)\widetilde \chi\big](x,y) \la y\ra^{-\sigma}\big\|_{L^2\to L^2} \\
&\qquad+ z^{2\sigma-1} \big\| (z+|x|)^{-\sigma} \chi(|x|)  \nabla\big[R_0(1)\widetilde \chi\big](x,y) \la y\ra^{-\sigma}\big\|_{L^2\to L^2}\\
&\qquad\qquad+ z^{2\sigma-1} \big\| \la x\ra^{-\sigma}   \nabla\big[R_0(1)\widetilde \chi\big](x,y) (z+|y|)^{-\sigma} \chi(|y|) \big\|_{L^2\to L^2},
\end{align*}
where $\chi$ is a smooth cutoff for $1/10$ neighborhood of the origin.
The first summand converges to zero provided that $\sigma>\frac12$ by the limiting absorption bound for the free Schr\"odinger operator for $\lambda=1$. 
Using the representation \eqref{eq:Hsplit} and the bound \eqref{eq:a}, and considering the Hilbert Schmidt norms, the second and third summands can be bounded by the square root of 
$$
z^{4\sigma-2}\int \la x\ra^{-2\sigma}\frac{1}{\la x-y\ra^{ n-1} } (z+|y|)^{-2\sigma} \chi(|y|) dx dy \les z^{0+} \to 0,
$$
provided that $\sigma>\frac12$.


\section{The high energies limiting absorption principle}
\label{sec:high_limap}

Let us briefly consider intermediate energies, i.e.,
$\lambda\in I:= [\lambda_0,\lambda_1]\subset (-\infty,-m) \cup (m,\infty)$. It was shown in \cite{GM}, see Theorem 1.6, that the resolvent of $H$ satisfies the limiting absorption principle uniformly in $\lambda$: 
\[
\sup_{\lambda\in I}\|\la x\ra^{-\sigma} 
\mR_V^\pm(\lambda) \la x\ra^{-\sigma}  \|_{L^2\to L^2} \le C_I,
\]
provided that there are no embedded eigenvalues, $\sigma>\frac12$, and $|V(x)| \les \la x\ra^{-1-}$. 

In this section we complete the proof of Theorem~\ref{thm:limap} by considering energies sufficiently far from threshold, in the non-compact interval $|\lambda| \in (\lambda_1, \infty)$. 
In other words, we establish a limiting absorption principle
for the perturbed Dirac resolvent $\mR_V^\pm (\lambda)$ at high energies:
\be\label{eqn:lap}
			\sup_{|\lambda|>\lambda_0}\| \mR_V^\pm (\lambda)\|_{L^{2,\sigma}\to L^{2,-\sigma}} \les 1,
			\qquad \sigma > \f12. 
\ee  
In fact we control the slightly stronger operator norm from $B$ to $B^*$, and show that embedded eigenvalues are absent in this part of the spectrum.

Recall that (with $z=\sqrt{\lambda^2-m^2}$) we have  
\be \label{eq:res_iden}
\mR_V^\pm (\lambda)= \mR_0^\pm(\lambda )\big[I + V\mR_0^\pm(\lambda)\big]^{-1}\\ 
= \mR_0^\pm(\lambda ) \big[I + L_z R_0^\pm(z^2)\big]^{-1},
\ee
where 
$$
L_z= V (D_m+\lambda).
$$
Scaling arguments nearly identical to Proposition~\ref{prop:R0scale} show that 
$\|\mR_0^\pm(\lambda)\|_{B \to B^*} \les 1$ for all $\lambda \geq \lambda_0 > m$.
 
Since multiplication by $V$ maps $B^*$ to $B$ (see \eqref{eq:BstarB}), the limiting absorption principle for $R_0$ implies a bound
$\|L_zR_0^\pm(z^2)\|_{B \to B} < C\|\la x\ra^{1+}V\|_{L^\infty}$ uniformly in $z \ge z_0 > 0$.  If $V$ is sufficiently small then the operator norm of $L_zR_0^\pm(z^2)$ is less than $1$ for all $z \geq z_0$.  Then one can conclude
$$
\sup_{z \geq z_0} \big\| \big[I + L_z R_0^\pm(z^2)\big]^{-1}\big\|_{B\to B} \lesssim 1.
$$ 

The main goal of this section is to show that the same bound on $[I + L_zR_0^\pm(z^2)]^{-1}$ holds even when $V$ is not
small. This cannot be proved directly from the size of $L_zR_0^\pm(z^2)$, which need not become small as $z \to \infty$.
Instead, the following crucial lemma shows that the Neumann series
\begin{equation}\label{eq:geomser}
\big[I + L_z R_0^\pm(z^2)\big]^{-1} = \sum_{\ell=0}^\infty (-1)^\ell
\big(L_z R_0^\pm(z^2)\big)^\ell
\end{equation}
is absolutely convergent for large $z$ due to the behavior of later terms in the series.

\begin{lemma} \label{lem:largem}
Assume the entries of  $V$ are continuous and satisfy $|V_{ij}(x)|\les \la x\ra^{-1-}$. There
exist sufficiently large $M = M( V)$ and $z_1 =
z_1( V)$ such that
\begin{equation}  \label{eq:smallprod}
  \sup_{z>z_1}\| (L_z R_0^\pm(z^2))^M \|_{B \to B} \le \f12.
\end{equation}
 \end{lemma}
 
Assuming for the moment that \eqref{eq:smallprod} holds, the operator inverse in \eqref{eq:res_iden} is bounded uniformly in $z > z_1$, thus we conclude \eqref{eqn:lap} for $\lambda_0=\lambda_0(V)$ sufficiently large.  
The remainder of this section is devoted to the proof of
Lemma~\ref{lem:largem}.    The method will be similar to the one in \cite{EGS2}.

\subsection{The directed free resolvent}

The first step is to decompose the free Schr\"odinger resolvent into a large number of pieces according to 
the size of $|x-y|$ and where $\frac{x-y}{|x-y|}$ lies on the unit sphere.  
This section presents a limiting absorption estimate for these truncated free
resolvent kernels and for their first-order derivatives. The constants will not depend
on the parameters of truncation, which gives us the freedom
to choose those values later on.
Similar estimates were obtained in \cite{EGS2} in dimensions $n\geq 3$, with derivatives of up to second order.
We emphasize here the steps where $n \geq 2$ and the number of derivatives are
most prominent, and refer the reader to \cite{EGS2} for technical details that are shared by
both arguments.

For any $\delta\in(0,1)$, let $\Phi_\delta$ be a smooth cut-off
function to a $\delta$-neighborhood of the north pole in $S^{n-1}$.
Also, for any $d\in(0,\infty)$,  $\eta_d(x)=\eta(|x|/d)$ denotes a smooth
cut-off to the set $|x|>d$. In what follows, we shall use the
notation
\[
R_{d,\delta}(\lambda^2)(x,y)=[R_0(\lambda^2) \eta_d \Phi_\delta]
(x,y) = R_0(\lambda^2)(x,y) \eta_d(|x-y|)
\Phi_\delta\Big(\frac{x-y}{|x-y|}\Big).
\]
Note that this operator obeys the same scaling as $R_0$,
see~\eqref{eq:Rscale}. More precisely,
\[
R_{d,\delta}(\lambda^2)(x,y) = \lambda^{n-2}
R_{d\lambda, \delta}(1)(\lambda x,\lambda y).
\]
Thus, Proposition~\ref{prop:R0scale}  applies to
$R_{d,\delta}(\lambda^2)$ in the form
\begin{equation}\label{eq:calRscale}
\|R_{d,\delta}(\lambda^2)\|_{B\to B^*} \les \lambda^{-1}
\|R_{d\lambda, \delta}(1)\|_{B\to B^*}
\end{equation}
for all $\lambda\ge1$ or, more generally,
\begin{equation}\label{eq:dercalRscale}
\| D^{\alpha}R_{d,\delta}(\lambda^2)\|_{B\to B^*} \les
\lambda^{-1+|\alpha|}
\|D^{\alpha}R_{d\lambda, \delta}(1)\|_{B\to B^*}
\end{equation}
for all multi-indices $\alpha$ and $\lambda\ge1$.

We sketch a proof of a limiting absorption
bound for $R_{d,\delta}$ and its derivatives of order at most one uniformly in the parameters $d,\delta\in(0,1)$, see
Proposition~\ref{prop:dir_res} below. This will be based on the
oscillatory integral estimate in Lemma~\ref{lem:VCP}, which was proved in \cite{EGS2} for $n\geq 3$ and $\frac{n-1}{2}\le p\le \frac{n+3}{2}$. The extension here to dimension $n=2$ with a smaller range of $p$ can be obtained from the proof of Lemma~3.4 in \cite{EGS2} by minor changes in the case analysis. More specifically, there is a step where one can replace the inequality $3(\frac{n-1}{2}) \geq \frac{n+3}{2}$ (which is true only if $n \geq 3$) with the slightly weaker bound $3(\frac{n-1}{2}) \geq \frac{n+1}{2}$. 
\begin{lemma}
  \label{lem:VCP} Let $\chi$  denote a
 smooth cut-off function to the
  region $1<|x|<2$. With $a(r)$ as in~\eqref{eq:a}, define
\begin{equation}\label{eq:pint}
(T_{\delta,p, R_1, R_2} f)(x) = \int \chi\big(\frac{x}{R_1}\big)
\frac{e^{i|x-y|}}{|x-y|^p}
   a(|x-y|)
  \Phi_\delta\Big(\frac{x-y}{|x-y|}\Big) \chi\big(\frac{y}{R_2}\big) f(y)\, dy.
\end{equation}
Then, for any $n\ge2$, and
  $\frac{n-1}{2}\le p\le \frac{n+1}{2}$,
\begin{equation}\label{eq:Tbound}
\| T_{\delta,p, R_1, R_2} f \|_2 \le C_n\, \delta^{p-\frac{n-1}{2}}
\sqrt{R_1 R_2}\, \|f\|_2
\end{equation}
  for all $R_1, R_2\ge1$, $\delta\in(0,1)$. The constant $C_n$ only
  depends on $n\ge2$.
\end{lemma} 

\begin{prop}
  \label{prop:dir_res}
  Let $n\ge2$. Then for any $d\in (0,\infty)$, $\delta\in (0,1)$, and
$\lambda\ge 1$ there is the bound
\begin{equation}\label{eq:limap}
\| D^{\alpha}R_{d,\delta}(\lambda^2)f \|_{B^*} \le
C_n\lambda^{-1+|\alpha|} \|f\|_B
\end{equation}
for any $0\le|\alpha|\le 1$. The constant $C_n$ depends only on the
dimension $n\ge2$.
\end{prop}
\begin{proof} In view of \eqref{eq:calRscale} and \eqref{eq:dercalRscale}
it suffices to prove this estimate  for $\lambda=1$.  We need to
prove that for any $0\le|\alpha|\le 1$
\begin{equation}\label{eq:star}
\| \chi(\cdot/R_1) D^\alpha R_{d,\delta}(1)\, \chi(\cdot/R_2) f
\|_2 \le C_n\, \sqrt{R_1 R_2}\, \|f\|_2
\end{equation}
where $R_1, R_2\ge 1$ are arbitrary.  We write
\begin{equation}\label{eq:T0T1}
R_{d,\delta}(1)=R_0^+(1) \eta_d\Phi_\delta   = T_0 + T_1
\end{equation}
where the kernels of $T_0, T_1$ are
\begin{equation}\label{eq:Tdef}\begin{split}
 T_0(x,y) &=
 \frac{b(|x-y|)}{|x-y|^{n-2}}
   \eta_d(|x-y|) \Phi_\delta(x,y),\\
T_1(x,y) &= \frac{e^{i|x-y|}}{|x-y|^{\frac{n-1}{2}}}
   \eta_d(|x-y|)a(|x-y|)
  \Phi_\delta(x,y),
  \end{split}
\end{equation}
respectively, see \eqref{eq:Hsplit}. The modified function $\eta_d(r)a(r)$
satisfies all decay estimates in~\eqref{eq:a} with constants independent of
the choice of $d$.

We begin by showing that
$\widehat{T_0 f} = m_0 \hat{f}$ where $|m_0(\xi)|\les \la
\xi\ra^{-1}$. This will imply~\eqref{eq:star} for~$T_0$. By
definition
\[
m_0(\xi) = \int_0^\infty\int_{S^{n-1}} rb(r)  \eta_d(r) e^{-ir
\omega\cdot\xi} \Phi_\delta(\omega) \,\sigma(d\omega) \,dr.
\]
Since $b(r)=0$ if $r>\f34$, $|m_0(\xi)|\les 1$. Hence we may assume
that $|\xi|\ge1$.  If $|\xi_n| \ge |\xi|/10$, then $|\omega \cdot \xi| \ges
|\xi|$ and
\[
|m_0(\xi)|\les \int_{S^{n-1}} \Phi_\delta(\omega) \la
\omega\cdot\xi\ra^{-2}\,\sigma(d\omega)\les \delta^{n-1} |\xi|^{-2},
\]
where we have used that
\[
\Big| \int_0^\infty e^{-ir\rho} rb(r)\eta_d(r)\chi(r)\,dr\Big |\les
\la \rho\ra^{-1}.
\]
This follows from \eqref{eq:bodd} and \eqref{eq:beven} after an 
integration  by parts. Now suppose that $|\xi_n| \le
|\xi|/10$. Set $\xi=|\xi| \hat{\xi}$ and change integration
variables as follows:
\begin{align*}
&\int_{S^{n-1}}\int_0^\infty rb(r)  \eta_d(r) \chi(r) e^{-ir|\xi|
\omega\cdot\hat{\xi}}\,dr\, \Phi_\delta(\omega) \,\sigma(d\omega)\\
& = \int_{\R^{n-1}} \int_0^\infty rb(r)  \eta_d(r) \chi(r)
e^{-ir|\xi| u_1}\,dr \,\tilde \Phi_\delta(u_1,\ldots,u_{n-1})\,
du_1du_2\ldots du_{n-1}\\
& = \delta^{n-2} \int_0^\infty \int_\R rb(r)  \eta_d(r) \chi(r)
e^{-ir|\xi| u_1}\Psi_\delta(u_1)\, du_1 dr,
\end{align*}
where $(u_1, \ldots, u_{n-1})$ is a parametrization of the support of
$\Phi_\delta$, aligning $u_1$ with $\hat{\xi}$.  The function $\Psi_\delta$
is a smooth cut-off supported on an interval of
length $\sim \delta$ resulting from the integration of $\tilde{\Phi}_\delta$.
Thus,
\[
|m_0(\xi)|\les \delta^{n-2}\int_0^1
 |\widehat{\Psi_\delta}(r|\xi|)|\,dr\les
\delta^{n-2}|\xi|^{-1}\| \widehat{\Psi_\delta}(u)\|_{L^1_u}\les
\delta^{n-2}|\xi|^{-1}.
\]
In conclusion, $|m_0(\xi)|\les \la \xi\ra^{-1}$ as claimed.

Next, consider $T_1$.  By the Leibniz rule,
\begin{align}
D^\alpha_x T_1(x,y) &= \sum_{\beta\le \alpha} c_{\alpha,\beta}\,
D^{\alpha-\beta}_x \Big[\frac{e^{i|x-y|}}{|x-y|^{\frac{n-1}{2}}}
  \eta_d(|x-y|) a(|x-y|)\Big] D^{\beta}_x
  \Phi_\delta(x,y) \nonumber\\
  & = \sum_{\beta\le \alpha} \delta^{-|\beta|}\,c_{\alpha,\beta}\,
\frac{e^{i|x-y|}}{|x-y|^{\frac{n-1}{2}+|\beta|}}
  a_{\alpha,\beta, d}(|x-y|)
  \Phi_{\delta,\beta}(x,y),\label{eq:newT1}
\end{align}
where $\Phi_{\delta,\beta}= \delta^{|\beta|} D^\beta \Phi_\delta$ is a
modified angular cut-off and $a_{\alpha,\beta, d}$ satisfies the same bounds
as $a$, see~\eqref{eq:a}, with constants that do not depend on $d$.
The estimate~\eqref{eq:star} for $T_1$ follows from Lemma~\ref{lem:VCP}
with $p=\frac{n-1}{2}+|\beta|$.
\end{proof}
 A partition of unity $\{\Phi_i\}$
over $S^{n-1}$ induces a directional decomposition of the free
resolvent, namely
\begin{equation} \label{eq:decomp}
R_0(\lambda^2) = \sum_{i} R_i(\lambda^2) + R_d(\lambda^2)
\end{equation}
where $R_i(\lambda^2) := R_{d,\delta}(\lambda^2)$ with
$\Phi_i$ playing the role of $\Phi_\delta$ from above. Moreover, $R_d(\lambda^2)(x)=(1 - \eta_d(|x|))R_0(\lambda^2)(|x|)$ is
the ``short range piece''.  We have the following $L^2$ bound for   $R_d(\lambda^2)$:

\begin{lemma} \label{lem:R_d}
With $R_d^+(\lambda^2)$ defined as above, the mapping estimate
\begin{equation} \label{eq:R_d}
\|D^\alpha R_d(\lambda^2)f \|_2 \le C_n\, \lambda^{-2+|\alpha|}
 \la d\lambda\ra \|f\|_2
\end{equation}
holds uniformly for every choice of $d \in (0,\infty)$, $0\le|\alpha|\le1$,
and $\lambda \ge 1$.
\end{lemma}
\begin{proof} By the scaling relation \eqref{eq:Rscale}, for any $\alpha$,
\[
\| D^\alpha R_0(\lambda^2) \chi_{[|x|<d]} \|_{2\to 2}= \lambda^{-2+|\alpha|} \|
D^\alpha R_0(1) \chi_{[|x|<\lambda d]} \|_{2\to2}
\]
where $\chi_{[|x|<\rho]}=\chi(|x|/\rho)$ is a smooth cut-off to the
set $|x|<\rho$ with $\rho>0$ arbitrary.  The notation is somewhat
ambiguous here; we are seeking an estimate for the convolution
operator with kernel $D^\alpha R_0(1)\chi_{[|x| < \lambda d]}$.  The
lemma is proved by showing that the Fourier transform of
$R_0(1)\chi_{[|x| < \rho]}$ is bounded point-wise by
$\japanese[\rho]\japanese[\xi]^{-1}$.

Consider first the case $\rho \le 1$.  The decomposition~\eqref{eq:Hsplit}
implies that $$\int_{\R^n} |R_0(1)(x)\chi(|x|/\rho)|\,dx \les \rho^2\log(\rho).$$
Furthermore, since $(\Delta+1)R_0(1)$ is a point mass at the origin,
the distribution $\Delta[R_0(1)\chi_{[|x| < \rho]}]$ consists of a point mass
plus a function of  $L^1$ norm $\lesssim |\log(\rho)|$.  This implies that the Fourier transform of $R_0(1)\chi_{[|x| < \rho]}$ is bounded     by  $|\log(\rho)| |\xi|^{-2}$.  The desired Fourier transform
estimate follows by interpolating these two bounds.   

When $\rho > 1$, it is more convenient to estimate
\[
\rho^n \Big| \int \big[\PV \frac{1}{|\eta|^2-1} +
i \sigma_{S^{n-1}}(d\eta)\big]\hat{\chi}((\xi-\eta)\rho)\, d\eta  \Big|.
\]
A standard calculation shows this to be less than
$\rho\japanese[\rho(|\xi|^2 - 1)]^{-1} < \rho\japanese[\xi]^{-2}$.
\end{proof}
 
\subsection{Proof of Lemma~\ref{lem:largem}}
Decomposing each free resolvent in the $M$-fold product
$(L_zR_0(z^2) )^M$ as in \eqref{eq:decomp} yields the identity
\begin{equation} \label{eq:decomp2}
(L_zR_0(z^2) )^M = \sum_{i_1 \ldots i_M} \prod_{k=1}^M
  \big(L_z R_{i_k}(z^2) \big).
\end{equation}
The indices $i_k$ may take numerical values corresponding to the
partition of unity $\{\Phi_i\}$, or else the letter $d$ to indicate
a short-range resolvent. There are two main types of products
represented here, namely:
\begin{itemize}
\item {\em Directed Products}, where the support of functions $\Phi_{i_k}$
and $\Phi_{i_{k+1}}$ are separated by less than $10\delta$ for each
$k$. A product is also considered to be directed if it has this
property once all instances of $i_k = d$ are removed.  The term
$(L_zR_d(z^2) )^M$ is a vacuous example of a directed product.

\item All other terms not meeting the above criteria are
 {\em Undirected Products}.  An undirected product must contain two
adjacent numerical indices (i.e., after discarding all instances
where $i_k = d$) for which the corresponding functions $\Phi_i$ have
disjoint support with distance at least $10\delta$ between them.
\end{itemize}

\begin{lemma} \label{lem:combinatorics}
For any $\delta>0$, there exists a partition of unity $\{\Phi_i\}$
with approximately $\delta^{1-n}$ elements, having ${\rm diam}\,{\rm
supp}\,(\Phi_i) < \delta$ for each $i$ and admitting no more than
$\delta^{1-n}(C_n)^M$ directed products of length $M$
in~\eqref{eq:decomp2}.
\end{lemma}
\begin{proof}
The first claim is a standard fact from differential geometry. For
the second claim note that there are $\les \delta^{1-n}$ choices for
the first element in a directed product, but only $C_n$ choices at
each subsequent step.
\end{proof}

The iterated resolvent $(L_zR_0(z^2))^M$ is an oscillatory integral
operator with phase $e^{iz\sum_{k=1}^M |x_{k} - x_{k-1}|}$,
where $x_0 = y$ and $x_M = x$.  Loosely speaking, there is a 
region of stationary phase where $\sum_{k=1}^M |x_k - x_{k-1}|
\approx |x-y|$.  The integral kernel of a directed product is supported here, hence
one cannot gain any benefit from oscillation as $z \to \infty$ beyond the bounds
for individual resolvents in Proposition~\ref{prop:dir_res} and
Lemma~\ref{lem:R_d}.  Those bounds do not decrease to zero
in the limit of large $z$.  It appears that the operator norm of a
directed product does not decrease to zero either.

Never the less, one can show that long directed products 
have an operator norm that is small enough for our purposes.
The geometric idea is relatively simple:
If $\delta < \frac{1}{20M}$, then all the angular cutoffs
$\Phi_{i_k}$ have support within a single hemisphere. The
convolution operators $R_{i_k}(z^2)$ are therefore biased
consistently to one side.  This introduces a gain from the
product $\prod_{k=1}^M V(x_k)$, as only a handful of $x_k$
can be located near the origin, and $V(x_k)$ is small everywhere else.

The following lemma is adapted from Lemma~4.8
of~\cite{EGS2}.  There is a small but significant difference in the structure of
the perturbation.  Here $L_z =  V(D_m + \lambda)$ has the property that
$L_zR_0(z^2)$ is a bounded operator on the space $B$.  The symmetrized
version $(L_z + L_z^*)R_0(z^2)$ does not map $B$ to itself unless $V(x)$ is assumed to be differentiable.
This explains the use of more elaborate function spaces in~\cite{EGS2}
and the need for bounds on the second derivative of the truncated free resolvent
(which ultimately restricts the dimension to $n \geq 3)$.

Now we may state the bound for directed products involving $L_z$
in dimensions $n \geq 2$.

\begin{lemma} \label{lem:directedprod} 
Given any $r > 0$, there exists a distance $d = d(r) > 0$ such that
each directed product in~\eqref{eq:decomp2} satisfies the estimate
\begin{equation}
\bignorm[\prod_{k=1}^M \big(L_z R_{i_k}(z^2) \big)f][B]
\le
 C_{n,V,r}\, r^{M} \norm[f][B]
\end{equation}
uniformly over all $z> d^{-1}$ and all choices of $M$ and $\delta$
satisfying $\delta \le \frac{1}{20M}$.

Consequently, given any $c>0$, there exists a number $M= m(c,V)$ and a
partition of unity governed by $\delta = \frac{1}{20M}$ so that
the sum over all directed products achieves the bound
\begin{equation*} \label{eq:directedprod}
\sum_{\substack{i_1 \ldots i_M \\ directed}} \bignorm[\prod_{k=1}^M
\big(L_zR_{i_k}(z^2))][B \to B] \le
{\textstyle \frac{c}{2}}
\end{equation*}
uniformly in $z > d^{-1}$.
\end{lemma}

\begin{proof}
In this proof, we will keep track of the superscripts
$\pm$ on the resolvents. Also, we will write $\|V\|_{B^*\to B}=C_V$.
There is no loss of generality if we assume that $r < C_nC_V$.

After a rotation, we may assume that every function $\Phi_{i_k}$
which appears in the product has support within a half-radian
neighborhood of the north pole, where $x_n > \frac23$. If $f \in B$
is supported on the half space $\{x_n
> A\}$, then the support of $R_{i_k}^+(\lambda^2)f$ must be
translated upward to $\{x_n > A + \frac23d\}$.
The short-range resolvent $R_d^+(\lambda^2)$ does
not have a preferred direction; however if $f \in B$ is
supported on $\{x_n > A\}$ then ${\rm supp}\,R_d^+(\lambda^2)f
\subset \{x_n > A - 2d\}$.

The purpose of keeping track of supports is that if $f \in
B^* $ is supported away from the origin, in the set $\{|x| >
A\}$, then the estimate in \eqref{eq:BstarB} can be improved to
$$
\norm[Vf][B] \les  \norm[\la x \ra^{1+}V\chi_{[|x| > A]}][L^\infty]
\norm[f][B^*]\les \la A\ra^{0-} \norm[f][B^*].
$$
Note that we can choose $A=A(n,V,r)$ so that
\begin{equation} \label{eq:distantL}
\norm[Vf][B]\leq \frac{r^2}{C_n^2C_V}\|f\|_{B^*},
\end{equation}
provided that $f$ is supported  in the set $\{|x| >
A\}$.

Let $\chi$ be a smooth function supported on the interval $[-1,
\infty)$ such that $\chi(x_n) + \chi(-x_n) = 1$.  We will initially
estimate the operator norm of $\big(\prod_k (L_zR_{i_k}^+(z^2))
\big) \chi(x_n)$. Multiplication by $\chi(x_n)$ is bounded operator
of approximately unit norm on $B^*$.

The support of $\chi(x_n) f$ lies in the half-space $\{x_n > -1\}$.
Suppose every one of the indices $i_k$ is numerical.  Then each
application of an operator $L_zR_{i_k}^+(z^2)$ translates the
support upward by $\frac23d$. For the first $\frac{3A}{2d}$ steps the
operator norm of $L_zR_{i_k}^+(z^2)$ is bounded by $C_nC_V$.   Thereafter it is possible to use the stronger bound
of~\eqref{eq:distantL}  because the support will
have moved into the half-space $\{x_n > A\}$.  The combined estimate is
\begin{equation} \label{eq:alldirected}
\begin{aligned}
\bignorm[\prod_{k=1}^m (L_zR_{i_k}^+(z^2))\,\chi(x_n)f][B]
&\le (C_nC_V)^{m} \Big(\frac{r^2}{(C_nC_V)^2}\Big)^{m-\frac{3a}{2d}}
\norm[f][B] \\
& = (C_n C_V)^{-m}(r^{-1}C_n C_V)^{\frac{3a}{d}} r^{2m} \norm[f][B].
\end{aligned}
\end{equation}
This is also valid for small $m$ by our assumption that $r < C_nC_V$.

If each directed resolvent $R_{\Phi_i}^+(\lambda^2)$ is seen as
taking one step forward, then the short-range resolvent
$R_d^+(\lambda^2)$ may take as many as three steps back.  Suppose a
directed product includes exactly one index $i_k = d$.  This will
have the most pronounced effect if it occurs near the beginning of
the product, delaying the upward progression of supports by a total
of $4$ steps.  In this case one combines
~\eqref{eq:distantL}, and Lemma~\ref{lem:R_d}
to obtain
\[
\bignorm[\prod_{k=1}^m (L_zR_{i_k}^+(z^2))\,\chi(x_n)f][B]
\le (C_nC_V)^{m} d \Big(\frac{r^2}{(C_nC_V)^2}\Big)^{m-(\frac{3A}{2d}+4)}
\norm[f][B].
\]
Notice that this estimate agrees with the one in
\eqref{eq:alldirected} up to
a factor of $d(r^{-1}C_nC_V)^8$.
By setting $d = d(r) = \big(\frac{r}{C_nC_V}\big)^8$,
the bound in \eqref{eq:alldirected} is
strictly larger.  Similar arguments yield the same result for any
directed product with one or more instances of the short-range
resolvent $R_d^+(\lambda^2)$.

To remove the spatial cutoff, write
\begin{equation*}
\prod_{k=1}^m L_zR_{i_k}^+(z^2) \ = \ \Big(\prod_{k=1}^{m/2}
L_zR_{i_k}^+(z^2)\Big)(\chi(x_n) + \chi(-x_n))
\Big(\prod_{k=\frac{m}2+1}^m L_zR_{i_k}^+(z^2)\Big).
\end{equation*}
Consider the $\chi(x_n)$ term.  By~\eqref{eq:alldirected}, the first half of
the product carries an operator norm bound of $(C_n\, C_V)^{-\frac{m}{2}}
(r^{-1}C_nC_V)^{\frac{3A}{d(r)}}r^m$. The second half
contributes at most $(C_n C_V)^{m/2}$. Put together, this product has an operator norm less than
$C_{n,V,r}\,r^m$, where $C_{n,V,r} = (r^{-1}C_nC_V)^{\frac{3A}{d(r)}}$.

The $\chi(-x_n)$ term has nearly identical
estimates, by duality. The adjoint of any directed resolvent
$R_{\Phi}^+(z^2)$ is precisely
$R_{\tilde{\Phi}}^-(z^2)$, with $\tilde{\Phi}$ being the
antipodal image of $\Phi$.  Because the order of multiplication is
reversed, one applies the geometric argument above to the adjoint operators $ R_{i_k}^-(z^2) L_z^*$ (modulo the
antipodal map).

According to Lemma~\ref{lem:combinatorics} there are at most $\delta^{1-n}
(C_n)^m$ directed products of length $m$.
To prove \eqref{eq:directedprod}, it therefore suffices to let
$r = \frac{1}{2C_n}$, and $\delta = \frac{1}{20m}$ so that the sum of the
operator norms of all
directed products is bounded by $20^{n-1} C_{n,V} m^{n-1}2^{-m}$.
This can be made smaller than $\frac{c}{2}$ by choosing $m$ sufficiently
large.
\end{proof}

As for the undirected products, recall that their defining feature
is the presence of adjacent resolvents $R_{i}^+(\lambda^2)$ oriented
in distinct directions.  The resulting oscillatory integral has no
region of stationary phase, and therefore exhibits improved bounds
at high energy provided the potential $V(x)$ is smooth. The following lemma follows from Lemma 4.9 in \cite{EGS2} by minor changes in the proof.

\begin{lemma} \label{lem:nonstatphase}
Let $\Phi_1$ and $\Phi_2$ be chosen from a partition of unity of
$S^{n-1}$ so that their supports are separated by a distance greater
than $10\delta$. Suppose $V \in C^\infty(\R^n)$ with compact
support. Then for each $j \ge 0$, and any $N\ge1$,
\begin{equation}\label{eq:Phi1Phi2}
\big\| L_zR_{d,\Phi_2}^+(z^2) (L_z R_d^+(z^2))^j L_z
R_{d,\Phi_1}^+(z^2) \big\|_{B\to B} =
\O(z^{-N})
\end{equation}
as $z\to\infty$ and similarly for $R^-(z^2)$.
\end{lemma}
 
Note that under the conditions of Lemma~\ref{lem:nonstatphase}
 each undirected product in~\eqref{eq:decomp2} satisfies the bound
\begin{equation} \label{eq:nonstatphase}
\bignorm[\prod_{k=1}^M \big(L_z R_{i_k}(z^2)\big)][B \to B]
 = \O(z^{-N})
\end{equation}
for any $N\ge1$.
We now show by approximation that vanishing still holds for merely
continuous $V$, but without any control over the rate.

\begin{lemma} \label{lem:RiemannLebesgue}
Let $\Phi_1$ and $\Phi_2$ be chosen as in
Lemma~\ref{lem:nonstatphase}. Suppose $V$ is a continuous function
with $V\in Y$.  Then each undirected product in~\eqref{eq:decomp2}
satisfies the limiting bound
\begin{equation} \label{eq:RiemannLebesgue}
\lim_{z \to \infty} \bignorm[ \prod_{k=1}^M
\big(L_z R_{i_k}(z^2)\big)][B\to B] = \ 0.
\end{equation}
\end{lemma}
\begin{proof}
For any small $\gamma > 0$, there exists a smooth approximation
$V_\gamma \in C^\infty(\R^n)$ of compact support so that $\norm[V -
V_\gamma][B^*\to B] < \gamma$ and $\|V_\gamma\|_{B^*\to B}<2\|V\|_{B^*\to B}$. Define the
operator $L_{z,\gamma}$ accordingly. We have 
\begin{equation*}
\bignorm[\prod_{k=1}^M \big(L_z R_{i_k}(z^2)\big) -
 \prod_{k=1}^M \big(L_{z,\gamma}R_{i_k}(z^2)\big)][B\to B]
\les \gamma (2\|V\|_{B^*\to B})^{M-1}
\end{equation*}
uniformly in $z\ge1$.  Thus, by \eqref{eq:nonstatphase},
\begin{equation*}
\limsup_{z \to \infty} \bignorm[\prod_{k=1}^M
\big(L_z R_{i_k}(z^2)\big)][B\to B] \les \gamma
(2\|V\|_{B^*\to B})^{M-1}.
\end{equation*}
Sending $\gamma\to0$ finishes the proof.
\end{proof}

\begin{proof}[Proof of Lemma~\ref{lem:largem}]  Lemma~\ref{lem:directedprod} provides a
recipe for selecting a value of $M$, together with a partition of
unity $\{\Phi_i\}$ and a short-range threshold $d$, so that the sum
over all directed products in~\eqref{eq:decomp2} will be an operator
of norm less than $\frac{1}{4}$.  This fixes the number of undirected
products as approximately $\delta^{M(1-n)} = (20 M)^{M(n-1)}$.  For
each of these, Lemma~\ref{lem:RiemannLebesgue} asserts that its
operator norm tends to zero as $\lambda \to \infty$.  The same is
true for the finite sum over all undirected products of length $M$.
In particular it is less than the directed product estimate provided $z >
z_1(M)$ is sufficiently large.
\end{proof}

\section{Extension to the complex plane}\label{sec:complex}
This section provides a short proof of Corollary~\ref{cor:e-free} via a perturbation
argument.  We first record a strong statement of continuity in the limiting
absorption principle.
\begin{prop} \label{prop:continuity}
Let $\lambda_n \to 1$ and $\eps_n \to 0^+$.  Then for each $\sigma > \frac12$,
\begin{equation*}
\lim_{n\to \infty}
\|R_0(\lambda_n + i\eps_n) - R_0^+(1)\|_{L^2_\sigma \to L^2_{-\sigma}} = 0.
\end{equation*}
\end{prop}
\begin{proof}
By the triangle inequality and scaling relations,
\begin{align*}
\|R_0(\lambda_n + i\eps_n) &- R_0^+(1)\|_{L^2_\sigma \to L^2_{-\sigma}} \\
&\leq  \|R_0(\lambda_n + i\eps_n) - R_0^+(\lambda_n)\|_{L^2_\sigma \to L^2_{-\sigma}}
+ \|R_0^+(\lambda_n) - R_0^+(1)\|_{L^2_\sigma \to L^2_{-\sigma}} \\
&\les \|R_0(1+i{\scriptstyle \frac{\eps_n}{\lambda_n}}) - R_0^+(1)\|_{L^2_\sigma \to L^2_{-\sigma}}
+  \|R_0^+(\lambda_n) - R_0^+(1)\|_{L^2_\sigma \to L^2_{-\sigma}}.
\end{align*}
Both of the differences in the last line converge to zero by the limiting absorption principle and continuity of the resolvent respectively.
\end{proof}

\begin{proof}[Proof of Corollary \ref{cor:e-free}]
By the standard resolvent identities~\eqref{eq:res_iden} and~\eqref{eqn:resolvdef}
\begin{align*}
\mR_V(\lambda + i\gamma) &=  \mR_0(\lambda+i\gamma) \big[I + V\mR_0(\lambda+i\gamma)\big]^{-1}\\
&= (D_m+\lambda+i\gamma)R_0((\lambda+i\gamma)^2 - m^2)
\big[I+V\mR_0(\lambda+i\gamma)\big]^{-1}
\end{align*}
The free Dirac resolvent is controlled by rescaling by $\lambda$.  For the gradient term,
$$
\| \nabla R_0((\lambda+i\gamma)^2-m^2)\|_{L^2_\sigma \to L^2_{-\sigma}}
\leq  \| \nabla R_0((1+i{\scriptstyle \frac{\gamma}{\lambda}})^2 - ({\scriptstyle \frac{m}{\lambda}})^2)\|_{L^2_\sigma \to L^2_{-\sigma}} \les 1,
$$
with the last inequality following from continuity of the Schr\"odinger resolvent at 1.  Similarly,
\begin{multline*}
|m+\lambda+i\gamma| \,\| R_0((\lambda+i\gamma)^2-m^2)\|_{L^2_\sigma \to L^2_{-\sigma}}\\
\leq \Big(1+\frac{m}{\lambda} + \frac{\gamma}{\lambda}\Big) 
\| R_0((1+i{\scriptstyle \frac{\gamma}{\lambda}})^2 - ({\scriptstyle \frac{m}{\lambda}})^2)\|_{L^2_\sigma \to L^2_{-\sigma}}
\les 1
\end{multline*}
provided $\frac{\gamma}{\lambda}$ and $\frac{m}{\lambda}$ are sufficiently small.  Boundedness of
$\mR_V(\lambda+i\gamma)$ therefore rests on the behavior of $[I+V\mR_0(\lambda+i\gamma)]^{-1}$.

The crucial bound~\eqref{eq:smallprod} shows that
$$
\| [I + V\mR_0^+(\lambda)]^{-1} \|_{L^2_\sigma \to L^2_{\sigma}} \les 1 + \| V\mR_0^+(\lambda) \|^{M-1}_{L^2_\sigma \to L^2_\sigma}
\leq C(V)
$$
for all $\lambda > \lambda_0$.  We would like to expand
\begin{equation} \label{eqn:gamma_perturbed}
\big[I + V\mR_0(\lambda+i\gamma)\big]^{-1} = \big[ (I + V\mR_0^+(\lambda)) + 
V\big(\mR_0(\lambda+i\gamma) - \mR_0^+(\lambda)\big) \big]^{-1}
\end{equation}
via a Neumann series, and this can be done provided the operator norm of
$V(\mR_0(\lambda+i\gamma) - \mR_0^+(\lambda))$ is less than $\frac{1}{C(V)}$.

The difference of Dirac resolvents can be expanded out as
\begin{align*}
\mR_0(\lambda+i\gamma) - \mR_0^+(\lambda)
&= (D_m + \lambda + i\gamma)R_0((\lambda+i\gamma)^2-m^2) - (D_m + \lambda)R_0^+(\lambda^2-m^2)
\\
&= i\gamma R_0^+(\lambda^2 - m^2) + (D_m + \lambda + i\gamma)
\big(R_0((\lambda+i\gamma)^2-m^2)-R_0^+(\lambda^2-m^2)\big).
\end{align*}
Each of these terms can be bounded using the same scaling arguments as above, with the end result
\begin{multline}
\|\mR_0(\lambda+i\gamma) - \mR_0^+(\lambda)\|_{L^2_\sigma \to L^2_{-\sigma}} \\
\les \frac{\gamma}{\lambda}\|R_0^+(1 - ({\scriptstyle \frac{m}{\lambda}})^2)\|_{L^2_\sigma \to L^2_\sigma} 
+ \Big(1 + \frac{m}{\lambda} + \frac{\gamma}{\lambda}\Big)
\|R_0((1 + i{\scriptstyle \frac{\gamma}{\lambda}})^2 - ({\scriptstyle \frac{m}{\lambda}})^2) 
- R_0^+(1-({\scriptstyle \frac{m}{\lambda}})^2) \|_{L^2_\sigma \to L^2_\sigma}.
\end{multline}

Proposition~\ref{prop:continuity} asserts the existence of constants $\delta > 0$ and $K < \infty$
 such that
if $0 \leq \frac{\gamma}{\lambda}, \frac{m}{\lambda} < \delta$, then
$\|R_0^+( 1 - (\frac{m}{\lambda})^2)\| \leq K$ and 
$\|R_0((1 + i\frac{\gamma}{\lambda})^2 - (\frac{m}{\lambda})^2) 
- R_0^+(1-(\frac{m}{\lambda})^2) \| \leq \frac{1}{10\|V\| C(V)}$.
The latter inequality holds because both arguments in the free resolvent 
reside in a small neighborhood of 1.  With the additional restrictions that
$\delta < \min(1, \frac{1}{5K \|V\|C(V)})$, it follows that
$$
\|\mR_0(\lambda+i\gamma) - \mR_0^+(\lambda)\|_{L^2_\sigma \to L^2_{-\sigma}} < \frac{1}{2\|V\|C(V)}.
$$
Composing with pointwise multiplication by $V$ completes the estimate for~\eqref{eqn:gamma_perturbed}.

When $|\gamma| > \|V\|$, the perturbed resolvent can be estimated by much more elementary means. Here we can use self-adjointness of $D_m$ to bound $\|\mR_0(\lambda + i\gamma)\|_{L^2 \to L^2} \leq \frac{1}{|\gamma|}$, after which it follows that the Neumann series for $[I + V\mR_0(\lambda + i\gamma)]^{-1}$ converges even in the space of bounded operators on $L^2$.  One concludes that $\|\mR(\lambda + i\gamma)\|_{L^2 \to L^2} \leq \frac{1}{|\gamma|(1-|\gamma|\, \|V\|)}$.

Within the cones $0 < |\gamma| < \delta |\lambda|$ with $|\lambda| > \lambda_1$, the perturbed resolvent can instead be bounded using~\eqref{eqn:limap2} and the identity
$$
\mR(\lambda +i\gamma) = \mR_0(\lambda + i\gamma) - \mR_0(\lambda+i\gamma)V\mR_0(\lambda + i\gamma)
+ \mR_0(\lambda+i\gamma)V\mR(\lambda+i\gamma) V\mR_0(\lambda + i\gamma).
$$
The first two terms are bounded operators on $L^2$ with norms less than $\frac{1}{|\gamma|}$ and $\frac{\|V\|}{|\gamma|^2}$ respectively.  In the third term we use the decay of $V$ to map $L^2$ to $L^2_\sigma$, or from $L^2_{-\sigma}$ to $L^2$, then the composition is again a bounded operator on $L^2$ with norm no greater than $\frac{C\|V\|^2}{|\gamma|^2}$.

Put together, the perturbed Dirac resolvent $\mR(\lambda + i\gamma)$ exists as a bounded operator on $L^2$ whenever
$|\gamma| > \|V\|$ or when $|\lambda| > \max(\lambda_1, \frac{\|V\|}{\delta})$ and $\gamma \not= 0$. 
\end{proof}

\end{document}